\documentclass[11pt,a4paper]{article}
\usepackage{graphicx}
\usepackage{amsmath}
\usepackage{graphicx}
\usepackage{verbatim}
\usepackage{color}
\usepackage{subfigure}
\usepackage{float}
\usepackage{authblk}

\usepackage{amsfonts}
\usepackage{textcomp}
\usepackage{eurosym}
\usepackage{enumerate}

\usepackage{dynkin-diagrams}
\tikzset{big arrow/.style={
		-Stealth,line cap=round,line width=1mm,
		shorten <=1mm,shorten >=1mm}}

\usepackage{textcomp}
\usepackage{eurosym}
\usepackage{array}
\usepackage{amssymb}
\usepackage{subfigure}
\usepackage{enumitem}

\usepackage{amsthm, amsmath}
\theoremstyle{plain}
\theoremstyle{definition}
\newtheorem{thm}{Theorem}
\newtheorem{cor}{Corollary}

\newtheorem{prop}{Proposition}

\newtheorem{lem}{Lemma}
\newtheorem{remark}{Remark}

\begin{document}
\title{The maximal subsemigroups of the ideals on a monoid of partial injections}
\author[1*]{Apatsara Sareeto}
\author[2]{Jörg Koppitz}

\affil[1*]{Corresponding author. Institue of Mathematics, University of Potsdam, Potsdam, 14476, Germany  (E-mail: channypooii@gmail.com)}
\affil[2]{Institute of Mathematics and Informatics, Bulgarian Academy of Sciences, Sofia, 1113, Bulgaria  (E-mail: koppitz@math.bas.bg)}

\maketitle 
\begin{abstract} In the present paper, a submonoid of the well studied monoid $POI_n$ of all order-preserving partial injections on an $n$-element chain is studied. The set $IOF_n^{par}$ of all partial transformations in $POI_n$ which are fence-preserving as well as parity-preserving form a submonoid of $POI_n$. We describe the Green's relations and ideals of $IOF_n^{par}$. For each ideal of $IOF_n^{par}$, we characterize the maximal subsemigroups. We will observe that there are three different types of maximal subsemigroups.


\end{abstract}

\section{Introduction and Preliminaries} 
Let $\overline{n}$ be a finite chain with $n$ elements (where $n$ is a positive integer), denoted as $\overline{n} = \{1 < 2 < \cdots < n\}$. We denote by $PT_n$ the monoid (under composition) of all partial transformations on $\overline{n}$. A partial transformation $\alpha$ on the set $\overline{n}$ is a mapping from a subset $A$ of $\overline{n}$ into $\overline{n}$. The domain(\mbox{respectively}, image) of $\alpha$ is denoted by $dom(\alpha) (\mbox{respectively}, \ im(\alpha))$. The empty transformation is symbolized as $\varepsilon$, and it is the transformation with $dom(\varepsilon) = im(\varepsilon) = \emptyset$. Let $Id_{\overline{n}}$ be the set of all partial identities on $\overline{n}$, where $id$ is the identity mapping on $\overline{n}$. A transformation $\alpha\in PT_n$ is called order-preserving if $x<y$ implies $x\alpha<y\alpha$ for all $x,y\in dom(\alpha)$. An injective $\alpha\in PT_n$ is called partial injection. The set $I_n$ (under composition) of all partial injections on $\overline{n}$ forms a monoid, referred to as a symmetric inverse semigroup, which was introduced by Wagner \cite{wagner}. We denoted by $POI_n$ the monoid of all partial order-preserving injections on $\overline{n}$. \\
\indent Recall, a subsemigroup $T$ of a semigroup $S$ is called maximal subsemigroup of $S$ if $T$ is contained  in no other proper subsemigroup of $S$. A left ideal of $S$ is a subset $I$ of $S$ such that $SI= \{sx :
s \in S, x \in I \} \subseteq I$. A right ideal is defined analogously, and an ideal of $S$ is a subset of $S$ that is both a left ideal and a right ideal. For more general background on semigroups and standard notations, we refer the reader to \cite{clif, Howie}. \\
\indent There has been a growing interest in the study of maximal subsemigroups within transformation semigroups. Notably, several researchers have made significant contributions.  In  \cite{Yan}, Yang characterized the maximal subsemigroups of the semigroup $O_n$ of all full order-preserving transformations. Dimitrova and Koppitz classified the maximal subsemigroups of the ideals of $O_n$ in \cite{I}. Ganyushkin and Mazorchuk   provided a  description of the maximal subsemigroups of the semigroup 
$POI_n$ in \cite{Maz}. Dimitrova and Koppitz offered a characterization of the maximal subsemigroups of the ideals of the semigroup $POI_n$ \cite{II}.   
In \cite{D}, Dimitrova and Mladenova explored the maximal subsemigroups of the semigroup of all partial order-preserving transformations.  Recently,  Zhao and Hu have determined both the maximal subsemigroups and the maximal subsemibands of the ideals of the monoid of all orientation-preserving and extensive full transformations on $\overline{n}$ \cite{Zuh}.  Additionally, in \cite{Gra}, Graham, Graham, and Rhodes have demonstrated that every maximal subsemigroup of a finite semigroup has certain features, and that every maximal subsemigroup must be one of a small number of types. As is often the case for semigroups, this classification depends on the description of maximal subgroups of certain finite groups. It is worth noting that maximal subsemigroups in many other families of transformation monoids have already been described or quantified, primarily through the work by Dimitrova, East, Fernandes, and other co-authors, as detailed in references such as \cite{fe, di, Jame,  ig} and the associated literature. \\
\indent A non-linear order that is closed to a linear order in some sense is the so-called zig-zag order. The pair $(\overline{n}, \preceq)$ is called a zig-zag poset or fence if 
\begin{center}
	$1 \prec 2 \succ \cdot\cdot\cdot \prec n - 1 \succ n$  if n is odd 
	and  $1 \prec 2 \succ \cdot\cdot\cdot \succ n - 1 \prec n$  if n is even, respectively.
\end{center}
The definition of the partial order $\preceq$ is self-explanatory. The number of order-preserving maps of fences and crowns, as well as transformations on fences, was first considered by Currie and Visentin \cite{currie} and Rutkowski \cite{Rutkowski}. The formula for the number of order-preserving self-mappings of a fence was also introduced by Rutkowski \cite{Rutkowski}. We observe that every element in a fence is either minimal or maximal, and for all $x,y\in \overline{n}$ with $x \prec y$, it follows $y\in \{x - 1, x + 1\}$.  We say that a transformation $\alpha \in I_n$ is fence-preserving if $x \prec y$ implies $x\alpha \prec y\alpha$, for all $x, y \in dom (\alpha)$. We denote by $PFI_n$ the submonoid of $I_n$ of all  fence-preserving partial injections on $\overline{n}$. Fernandes et al. determined the rank and a minimal generating set of the monoid of all order-preserving transformations on an $n$-element zig-zag ordered set \cite{fer 2}. It is worth mentioning that several other properties of monoids of fence-preserving full transformations were also studied in \cite{Dimi, Jen,  Jit, Loh,    Jane, Rat, Tanya}.
\indent  We denote by $IF_n$ the inverse subsemigroup of $PFI_n$ of all regular elements in $PFI_n$. It is easy to see that $IF_n$ is the set of all $\alpha\in PFI_n$ with $\alpha^{-1}\in PFI_n$.   In the present paper, we consider a submonoid of both monoids $IF_n$ and $POI_n$, i.e. a submonoid of $IOF_n=IF_n\cap POI_n$.  Let $a\in dom(\alpha)$ for some $\alpha\in IOF_n$. If $a+1\in dom(\alpha)$ or $a-1\in dom(\alpha)$ then it is easy to verify that $a$ and $a\alpha$ have the same parity, that is, $a$ is odd if and only if $a\alpha$ is odd. However, if $a-1$ and $a+1$ are not in $dom(\alpha)$, then $a$ and $a\alpha$ can have different parity. In order to exclude this case, we require that the image of any $a\in dom(\alpha)$ has the same parity as $a$. In this scenario, we refer to $\alpha$ as parity-preserving. Our focus lies on the set $IOF_n^{par}$ of all parity-preserving transformations in $IOF_n$. Notably, for any $\alpha\in IOF_n^{par}$, the inverse partial injection $\alpha^{-1}$ exists  and possesses order-preserving, fence-preserving, and parity-preserving. This observation implies that $IOF_n^{par}$ can be considered as an inverse submonoid of $I_n$, as explained in \cite{Apa}. \\
\indent In Section \ref{ideal}, we will repeat a charactrization of the elements within $IOF_n^{par}$. We introduce a relation, denoted as $\sim$, on the power set $\mathcal{P}(\overline{n})$ of $\overline{n}$. This relation offers an alternative   characterization the monoid $IOF_n^{par}$, and furthermore,  this characterization leads us to an immediate descriptions of the Green's relation $\mathcal{J}$. Note that the Green's relations $\mathcal{R}, \mathcal{L}$, and $\mathcal{H}$ are already known, given that $IOF_n^{par}$ is an inverse submonoid of $I_n$. However, this paper deals mainly with  the ideals of $IOF_n^{par}$. Of course, the sets $IOF_n^{par}$ and $\{\varepsilon\}$ are ideals of $IOF_n^{par}$, which are often referred to as the trivial ideals.
We will demonstrate that the ideals of $IOF_n^{par}$ are of the form: 
\begin{center}  $I_{P^*}=\{\alpha\in IOF_n^{par}: dom(\alpha)\in P^*\}$ \end{center}
for particular subsets $P^*\subseteq \mathcal{P}(\overline{n})$. For each ideal $I\neq \{\varepsilon\}$ of $IOF_n^{par}$, we characterize the maximal subsemigroups of $I$.  Our investigation will reveal that there are three distinct types of maximal subsemigroups within an ideal $I\neq \{\varepsilon\}$ of $IOF_n^{par}$.
Therefore, the characterization of the ideals of $IOF_n^{par}$ (Section \ref{ideal}) and the description of the maximal subsemigroups of these ideals (Section \ref{max}) constitute the main results of this paper.

\section{The ideals} \label{ideal}
In this section, we will describe the ideals of $IOF_n^{par}$. First, we will provide a characterization of the elements within $IOF_n^{par}$. For the sake of completeness, we will recall the proof of the following Proposition, which describes the partial injections in $IOF_n^{par}$.

\begin{prop} \label{4 choice} \cite{Apa}
	Let  $p\leq n$ and let $\alpha = \bigl(\begin{smallmatrix}
	d_1 &<&d_2&<&   \cdots & < &d_p \\
	m_1 & &m_2 & & \cdots &    & m_p
	\end{smallmatrix}\bigr) \in I_n$. Then $\alpha\in IOF_n^{par}$ if and only if  the following four conditions hold:\\
	(i) $m_1<m_2< \cdots <m_p$; \\
	(ii) $d_1$ and $m_1$ have the same parity; \\
	(iii) $d_{i+1}-d_i=1$ if and only if $m_{i+1}-m_i=1$ for all $i\in\{1,...,p-1\}$;  \\
	(iv) $d_{i+1}-d_i$ is even if and only if $m_{i+1}-m_i$ is even for all $i\in\{1,...,p-1\}$.
\end{prop}
\begin{proof}   $(\Rightarrow)$: (i) and (ii) hold since $\alpha$ is order- and parity-preserving, respectively. 
	(iii): Since $\alpha\in IF_n$, we have $d_{i+1}\alpha-d_i\alpha=1$, i.e. $m_{i+1}-m_i=1$, if and only if $d_{i+1}-d_i=1$, for all $i\in\{1,...,p-1\}$. 	(iv): Suppose  $d_{i+1}-d_i$ is even. Then $d_{i+1}$ and $d_i$ have the same parity. Moreover, $\alpha$ is parity-preserving. This implies $d_{i+1}\alpha$ and $d_i\alpha$ have the  same parity, i.e. $m_{i+1}-m_i$ is even. The converse direction can be proved dually. \\	
	
	\indent \quad $(\Leftarrow)$: By (i), we can conclude that $\alpha$ is order-preserving. Let $i\in\{1,...,p-1\}$ and suppose $d_i$ and $m_i$ have the same parity. Then $d_i-m_i=2k$ for some integer $k$. According to (iv), we have $(d_{i+1}-d_i)-(m_{i+1}-m_i)=2l$ for some integer $l$. We obtain $2l=d_{i+1}-m_{i+1}-(d_i-m_i)=d_{i+1}-m_{i+1}-2k$, i.e. $d_{i+1}-m_{i+1}=2(l+k)$. This implies that $d_{i+1}$ and $m_{i+1}$ have the same parity. Together with (ii), we can conclude that $\alpha$ is parity-preserving. Now, let  $x\prec y$. This provides $\lvert x- y \rvert=1$. We have $ \lvert x\alpha - y\alpha \rvert = 1 $ by (iii). Since $\alpha$ is parity-preserving, $ \lvert x\alpha - y\alpha \rvert = 1 $ and $x\prec y$ give $x\alpha \prec y\alpha$. Therefore, $\alpha\in PFI_n$. Similarly, we can demonstrate that $\alpha^{-1}\in PFI_n$, i.e., $\alpha\in IF_n$. Hence, we can conclude that $\alpha\in IOF_n^{par}$.
\end{proof}

A set $X\subseteq \mathcal{P}(\overline{n})$ is called convex if, for all $A,B\in X$ with $A\subseteq B$ and for all $C\in \mathcal{P}(\overline{n})$,  the following condition holds:  if $A\subseteq C\subseteq B$, then $C\in X$. Here, $\mathcal{P}(A)$ denotes the power set of $A$, for any $A\subseteq \overline{n}$. The following is easy to verify: \\ 

\begin{remark} \label{convex}
If the empty set is contained in $X\subseteq \mathcal{P}(\overline{n})$, i.e. $\emptyset\in X$, then $X$ is convex if and only if $\mathcal{P}(A)\subseteq X$ for all $A\subseteq X$. 
\end{remark}

\indent We will observe that the domains of all partial transformations within an ideal form a convex set with an additional requirement. In order to describe this requirement, we define a partial order $\sim$ on $\mathcal{P}(\overline{n})$.  Let $k_1, k_2\in \mathcal{P}(\overline{n})$ with $k_1=\{i_1 <i_2< \cdots <i_k\}$ and $k_2= \{j_1 <j_2< \cdots <j_l\}$ for some positive integers $k,l$. We put $k_1\sim k_2$, if the following three properties are satisfied:   \\
\indent (i) $k=l$; \\
\indent (ii) $i_r$ and $j_r$ have the same parity for all $r\in\{1,...,k\}$; \\
\indent (iii) $i_r-i_{r-1}=1$ if and only if $j_r-j_{r-1}=1$ for all $r\in\{2,...,k\}$. \\

It is worth mentioning that here (iii) (above) corresponds with (iii) in Proposition \ref{4 choice}. In fact, for $A,B\subseteq \overline{n}$, there is $\alpha\in IOF_n^{par}$ with $A=dom(\alpha)$ and $B=im(\alpha)$ if and only if $A\sim B$. So any $\alpha\in IOF_n^{par}$ is uniquely determine by domain and image, i.e. by $dom(\alpha)$ and $im(\alpha)$. In particular, we have $dom(\alpha)\sim im(\alpha)$ for any $\alpha\in IOF_n^{par}$. Using this description of the elements in $IOF_n^{par}$ , we obtain immediately the Green's relations $\mathcal{L}, \mathcal{R}, \mathcal{H}$, and $\mathcal{J}$ (see \cite{Howie}) for the inverse submonoid $IOF_n^{par}$ of $PT_n$: 
\begin{prop} Let $a,b\in IOF_n^{par}$ then \\
1) $a\mathcal{L}b$ if and only if $im(a)=im(b)$; \\
2) $a\mathcal{R}b$ if and only if $dom(a)=dom(b)$; \\
3) $a\mathcal{H}b$ if and only if $a=b$; \\
4) $a\mathcal{J}b$ if and only if $dom(a)\sim dom(b)$ (or $im(a)\sim im(b)).$

\end{prop}

\noindent For a set $X\subseteq \mathcal{P}(\overline{n})$ let 
\begin{center}  $I_X=\{\alpha\in IOF_n^{par}: dom(\alpha)\in X\}$. \end{center}
Now, we are able to characterize the ideals of $IOF_n^{par}$.

\begin{prop} \label{P star}
Any ideal $I$ of $IOF_n^{par} $ is of the form $I=I_{P^*}$, where $P^*$ is a convex subset of $\mathcal{P}(\overline{n})$ with $\emptyset\in P^*$ such that $y\in P^*$ implies $z\in P^*$ for all $z\in \mathcal{P}(\overline{n})$ with $z\sim y$.
\end{prop}
\begin{proof}
Let $P^*$ be a convex subset of $\mathcal{P}(\overline{n})$ as defined in the statement. Additionally, let $a\in I_{P^*}$ and $b\in IOF_n^{par}$. We observe that $im(a)\sim dom(a)$  by Proposition \ref{4 choice}. Consequently, $im(a)\in P^*$. Further, we have $dom(ab)\subseteq dom(a)$, i.e. $dom(ab)\in P^*$ by Remark \ref{convex}. This implies $ab\in I_{P^*}$. Moreover, we observe that $im(ba)\in P^*$ since $im(ba)\subseteq im(a)$ and Remark \ref{convex}. This gives $dom(ba)\in P^*$ since $dom(ba)\sim im(ba)$. So $ba\in I_{P^*}$, i.e. $I_{P^*}$ is an ideal. \\

Conversely, let $I$ be an ideal and let $P^*=\{dom(a):a\in I\}$. Note that any ideal contains the empty transformation $\varepsilon$. This gives $\emptyset\in P^*$. Now, let $A\in P^*$. This means that there exists $a\in I_{P^*}=I$ such that $A=dom(\alpha)$. Let $B\subseteq A$ and $c\in Id_{\overline{n}}$ with $dom(c)=B$. Then $c\in IOF_n^{par}$, where $B=dom(ca)$ and we observe that $ca\in I$. This provides $B=dom(ca)\in P^*$. So, we have shown that  $\mathcal{P}(A)\subseteq P^*$. Thus, $P^*$ is a convex set as established in Remark \ref{convex}. \\
\indent Let $y\in P^*$. Then there is $a\in I$ with $dom(a)=y$. Further, let $z\in \mathcal{P}(\overline{n})$ with $z\sim y$. Then, there is $b\in IOF_n^{par}$ with $dom(b)=z$ amd $im(b)=y$. We get that $dom(ba)=dom(b)$ and $ba\in I$ because of $a\in I$. So $z=dom(b)=dom(ba)\in P^*$. It clear that $I\subseteq I_{P^*}$  by the definition of the sets $P^*$ and $I_{P^*}$. Let $b\in I_{P^*}$. Then there is $\gamma\in I$ with $dom(b)=dom(\gamma)$ and then  $im(\gamma)\sim im(b)$. Furthermore, there are $\alpha_1,\alpha_2\in IOF_n^{par}$ with $dom(\alpha_1)=dom(b),im(\alpha_1)=dom(\gamma), dom(\alpha_2)=im(\gamma)$, and $im(\alpha_2)=im(b)$. So, we see that $\alpha_1\gamma\alpha_2=b$, i.e.  $b\in I$. Consequently, $I=I_{P^*}$. 
\end{proof}
 
  Clearly, $\{\emptyset\}$ is a convex set and  $I_{\{\emptyset\}}=\{\varepsilon\}$ is the least ideal of $IOF_n^{par}$. Now, we determine the minimal ideals of $IOF_n^{par}$. An ideal $I$ is called minimal if $I \neq \{\varepsilon\}$ and $M\subseteq I$ implies $M=I$, for all non-trivial ideals $M$ of $IOF_n^{par}$.

\begin{prop} Let $I$ be a non-trivial ideal of $IOF_n^{par}$. Then $I$ is a minimal ideal of $IOF_n^{par}$ if and only if $I=I_{P^*}$, where either $P^*=\{\emptyset, \{1\}, \{3\},...,\{n-1\}\}$ or $P^*=\{\emptyset,\{2\},\{4\},...,\{n\}\}$ if $n$ is even and either $P^*=\{\emptyset, \{1\}, \{3\},...,\{n\} \}$ or $P^*=\{\emptyset,\{2\},\{4\},...,\{n-1\}\}$ if $n$ is odd, respectively.
\end{prop}
\begin{proof} Without loss of generality, we can assume that $n$ is even. The proof for $n$ is odd is similar. \\
\indent Note that $\{1\}\sim \{3\}\sim \cdots \sim \{n-1\}$ and $\{2\}\sim \{4\}\sim \cdots \sim \{n\}$. By Proposition \ref{P star}, we get that $I_{P^*}$ is a minimal ideal, whenever $P^*=P^O=\{\emptyset, \{1\}, \{3\},...,\{n-1\} \}$ or $P^*=P^E=\{\emptyset,\{2\},\{4\},...,\{n\}\}$. \\

Conversely, let $I$ be a minimal ideal of $IOF_n^{par}$. Then, there is $P^* \subseteq \mathcal{P}(\overline{n})$, satisfying the conditions in Proposition \ref{P star}, such that $I=I_{P^*}$.  Assume  there is $a\in I$ such that $rank(a)\geq 2$.  This provides that there exists $b\in I$ with $rank(b)=1$ and $dom(b)\in \mathcal{P}(dom(a))$. Let $I'=\{b'\in I : rank(b')\leq1\}$. It is obvious that $I'$ is an ideal. This gives $\{\varepsilon\}\neq I'\subset I$, a contradiction to $I$ is a minimal ideal. So, $rank(a)\leq 1$ for all $a\in I$. We have now $P^* = \{dom(a) : a\in I\}\subseteq \{\emptyset, \{1\}, \{2\},...,\{n\}\}$.   Assume $P^*\neq P^O$ and $P^*\neq P^E$.  Moreover, assumex $P^E\cap P^*\neq\{\emptyset\}$. Then $P^E\subseteq P^*$ since $P^*$ satisfies the conditions in Proposition \ref{P star}. Since $I_{P^*}$ as well as $I_{P^E}$ are minimal ideals, we obtain $P^*=P^E$, a contradiction to $P^*\neq P^E$. Hence $P^E\cap P^*=\{\emptyset\}$. Similarly, we have $P^O\cap P^*=\{\emptyset\}$. But $P^E\cap P^*=P^O\cap P^*=\{\emptyset\}$ gives $P^*=\{\emptyset\}$, a contradiction to $I_{P^*}\neq \{\varepsilon\}$. Thus, $I=I_{P^*}$ with $P^*=P^O$ or $P^* = P^E$.  \end{proof}

\section{The maximal subsemigroups of the ideals on $IOF_n^{par}$} \label{max}

In this section, we determine the maximal subsemigroups of the ideals on $IOF_n^{par}$. First, we need a few technical tools. Let $y=\{i_1<i_2<\cdots< i_k\}\in P^*$  for some positive integer $k\geq 2$ and let $t\in\{1,2,...,k\}$. Then we put $y^{[t]}=y\backslash \{i_t\}$.  For $x,y\in P^*$, we write $x\sqsubset y$ if there is $t\in \{1,...,|y|\}$, such that $x=y^{[t]}$. Otherwise, we write $x\not\sqsubset y$. \\

\begin{lem} \label{dom,im}
Let $y_1,y_2,z_1,z_2\in P^*$ with $|y_1|=|z_1| \geq 2$ and let $r,s\in\{1,2,...,|y_1|\}$ such that  $y_1^{[r]}=z_1^{[s]}= y_1\cap z_1$,  $ y_2\sim y_1$, $z_2\sim z_1$, and $y_2^{[r]}\sim z_2^{[s]}\sim y_1\cap z_1$. Then there are $\theta,\delta\in I_{P^*}$ with $rank (\theta)=rank(\delta)=|y_1|$ such that $dom(\theta\delta)=y_2^{[r]}$ and $im(\theta\delta)=z_2^{[s]}$. 
\end{lem} 
\begin{proof}
Let $y_1,y_2,z_1,z_2\in P^*$ with $|y_1|=|z_1| \geq 2$ and let $r,s\in\{1,2,...,|y_1|\}$ such that  $y_1^{[r]}=z_1^{[s]}= y_1\cap z_1$, $ y_2\sim y_1$, $z_2\sim z_1$, and $y_2^{[r]}\sim z_2^{[s]}\sim y_1\cap z_1$. Then there are $\theta,\delta\in I_{P^*}$ with $dom(\theta)= y_2, im(\theta)=y_1, dom(\delta)=z_1$, and $im(\delta)= z_2$. Because of $y_1^{[r]}=z_1^{[s]} = y_1\cap z_1$, then $rank(\theta\delta) = |y_1|-1  $. Because of $ y_2\sim y_1$  and $z_2\sim z_1$ with $y_2^{[r]}\sim z_2^{[s]}\sim y_1\cap z_1$, then $dom(\theta\delta)=y_2^{[r]}$ and $im(\theta\delta)=z_2^{[s]}$. 
\end{proof}

Let $I_{P^*}^u$ be the set of all $\alpha\in I_{P^*}$ with $dom(\alpha)\neq  y_2^{[r]}$ or $im(\alpha)\neq z_2^{[s]}$, whenever $y_2, z_2\in P^*$ and  $r,s\in\{1,2,...,|y_1|\}$ with $ y_2\sim y_1, z_2\sim z_1, y_1^{[r]}=z_1^{[s]}= y_1\cap z_1$, and  $y_2^{[r]}\sim z_2^{[s]}\sim y_1\cap z_1$ for some $y_1,z_1\in P^*$ with  $|y_1|=|z_1| \geq 2$. Directly from Lemma \ref{dom,im}, we obtain:  $\alpha\notin I_{P^*}^u$ if and only if there are $\theta, \delta\in I_{P^*}$ with $rank(\theta)=rank(\delta)=rank(\theta\delta)+1$ such that $\alpha =\theta\delta$. Since $P^*$ is a convex subset of $\mathcal{P}(\overline{n})$, we can conclude:
\begin{cor} \label{>} 
Let $\alpha\in I_{P^*}$. Then $\alpha\notin I_{P^*}^u$ if and only if there are $\theta,\delta\in I_{P^*}$ with $rank(\theta), rank(\delta)>rank(\theta\delta)$ such that $\alpha=\theta\delta$.

\end{cor}  
Our initial observation is that all maximal subsemigroups of an ideal $I$ have the form $I\backslash T$, where all transformations in $T$ have the same rank.

\begin{lem} \label{beta}
Let $J$ be a maximal subsemigroup of $I_{P^*}$ and let $\alpha\notin J$. Then $\beta\in J$ for all $\beta\in I_{P^*}$ with $rank (\beta)\neq rank (\alpha)$.
\end{lem}
\begin{proof}
Assume there is $\beta\in I_{P^*}$ such that $\beta\notin J$ with $rank (\beta)\neq rank (\alpha)$. Suppose $rank (\beta) > rank (\alpha)$. We observe that $\langle J,\alpha \rangle$ is semigroup, where $ rank(a\alpha), rank(\alpha a) \leq rank(\alpha)$ for all $a\in J$ and we see that $\langle J,\alpha \rangle \neq I_{P^*}$ since $\beta\in I_{P^*}$ but $\beta \notin \langle J,\alpha \rangle$. Moreover, $J\subset  \langle J,\alpha \rangle\neq I_{P^*}$, a contradiction to $J$ is a maximal subsemigroup of $I_{P^*}$. Suppose $rank (\beta) < rank (\alpha)$. Then we can show by contradiction that $J\subset  \langle J,\beta \rangle\neq I_{P^*}$ in the same way. 
\end{proof}

  \noindent For the remainder of this section, let
  $P^*$ be a convex subset of $X$, satisfying the conditions in  Proposition \ref{P star}, with $I_{P^*}\neq \{ \varepsilon \}$. Now, we determine several subsemigroups of $I_{P^*}$ and will show that they are exactly the maximal ones. \\
 Let 
  $C=\{  \{g\}: g\in P^*, g\sqsubset y \ \text{for some} \ y\in P^* \text{with} \ y\nsim z \ \text{for all} \ z\in P^*\backslash \{y\} \ \text{or} \ g\not\sqsubset z \ \text{for all} \ z\in P^*\} $. \\
 For  $\Gamma_1\in P^*$ with $|\Gamma_1|\geq 2$, we define  $A_{\Gamma_1}=\{\Gamma\in P^* : \Gamma\sim \Gamma_1\}$ and \\ $B_{\Gamma_1}=\{\Gamma^t : \Gamma\in A_{\Gamma_1}, t\in\{1,2,...,|\Gamma|\}$.  \\
 We put $A=\{A_{\Gamma}: \Gamma\in P^*, |\Gamma|\geq 2, |A_{\Gamma}|\geq 2 \}$ and  $B=\{B_{\Gamma}: \Gamma\in P^*, |\Gamma|\geq 2, |A_{\Gamma}|\geq 2 \}$.  \\
Let $\Delta_1\in B\cup C$. We define \begin{center} $BC_{\Delta_1} = \{\Delta\in B \cup C : x\sim y \ \text{for all} \ x\in \Delta, y\in \Delta_1\}$. \end{center} 
 For  $\Delta\in BC_{\Delta_1}$, we also define \begin{center} $T_{\Delta}=\{ c \in I_{P^*} :  dom(c),im(c)\in \Delta\}$. \end{center} 
 And  for a partition $Q=(Q_1,Q_2)$ of $BC_{\Delta_1}$, let \begin{center} 
 $T_{Q}=\{c\in I_{P^*} :  dom(c)\in \tilde{\Delta},im(c)\in \hat{\Delta} \ \text{with} \ \tilde{\Delta} \in Q_1,  \hat{\Delta}\in Q_2 \}$. \end{center}

 \begin{lem} \label{alp} Let $g\in P^*$ with $\{g\}\in C$ and $|BC_{\{g\}}|=1$. Then 
 $I_{P^*}\backslash T_{\{g\}}$ with $T_{\{g\}}\subseteq I_{P^*}^u$  is a semigroup.
 \end{lem}
 
 \begin{proof} Let $\alpha\in T_{\{g\}}$. Because of $\{g\}\in C$ and $|BC_{\{g\}}|=1$, we get that $\alpha\in Id_{\overline{n}}$ and $T_{\{g\}}=\{\alpha\}$. Let $a,b\in I_{P^*}\backslash \{\alpha\}$.  If $rank(a)\leq rank(\alpha)$ or $rank(b)\leq rank(\alpha)$ then $ab\neq \alpha$ because  $dom(a)\nsim dom(\alpha)$ and $dom(b)\nsim im(\alpha)$.  If $rank(a),rank(b)>rank(\alpha)$ and $rank(ab)=rank(\alpha)$ then by Corollary \ref{>}, we have $ab \neq \alpha$ since $\alpha\in I_{P^*}^u$. This shows $ab\in I_{P^*}\backslash \{\alpha\}$.  
 Consequently, $I_{P^*}\backslash \{\alpha\}$ is a semigroup. 
 \end{proof} 
 
 \begin{lem} \label{yyy} Let $y_1\in P^*$ with $|y_1|\geq 2$, $B_{y_1}\in B $, and  $|BC_{B_{y_1}}|=1$. Then 
 $I_{P^*}\backslash T_{B_{y_1}}$ with $T_{B_{y_1}}\subseteq I_{P^*}^u$ is a semigroup.
 \end{lem}
 \begin{proof} Let $\alpha\in T_{B_{y_1}}$ and let $a, b\in I_{P^*}\backslash T_{B_{y_1}}$. Note that if $rank(a)\leq rank(\alpha)$ or $rank(b)\leq rank(\alpha)$ then we can immediately deduce that $dom(ab)\nsim dom(\alpha)$. Thus, $ab\notin B_{y_1}$ and $ab\in I_{P^*}\backslash T_{B_{y_1}}$. Suppose $rank(a),rank(b)>rank(\alpha)$ and $rank(ab)=rank(\alpha)$. Then by Corollary \ref{>}, we have $ab\notin I_{P^*}^u$. This shows $ab\in I_{P^*}\backslash T_{B_{y_1}}$. Consequently, $I_{P^*}\backslash T_{B_{y_1}}$ is a semigroup.   
 \end{proof}
 
 \begin{lem} \label{bc} Let $Q=( Q_1, Q_2 )$ be a partition of $BC_{\Delta_1}$, where 
 $\Delta_1\in B\cup C$. Then $I_{P^*}\backslash T_{Q}$ with $T_Q\subseteq I_{P^*}^u$ is a semigroup. 
 \end{lem}
 \begin{proof}  Let $a,b\in I_{P^*}\backslash T_{Q}$ and let $c\in T_Q$.  If $rank(a)<rank(c)$ or $rank(b)<rank(c)$ then $rank(ab)<rank(c)$. So, $ab\in I_{P^*}\backslash T_{Q}$.  If $rank(a)=rank(b)=rank(c)$ then we need to consider only the case that  $dom(a)\in \Delta'$ for some $\Delta'\in Q_1$ and $im(b)\in \Delta''$ for some $\Delta''\in Q_2$.  
  We get that $im(a)\in Y$ for some $Y\in Q_1$ and $dom(b)\in Y'$ for some $Y'\in Q_2$ because $a,b\in I_{P^*}\backslash T_Q$. We see that $rank(ab)<rank(c)$ because $ Q_1 \cap Q_2= \emptyset$. Suppose that  $rank(a)>rank(b)=rank(c)$   and  $rank(ab)= rank(c)$. Note if $im(b)\in \Delta\in Q_1$ then $im(ab)\in \Delta \in Q_1$. So, $ab\in I_{P^*}\backslash T_{Q}$. Suppose $im(b)\in \Delta$ for some $\Delta\in Q_2$. Then $dom(b) \in  \Delta' $ for some $\Delta'\in Q_2$ because of $b\in I_{P^*}\backslash T_{Q}$ and we have $dom(b)\subset im(a)$. We put $dom(a)=y_1, im(a)=y_2$, and $dom(b)=y_2^{[t]}$ for some $t\in\{1,2,...,|y_2|\}$. This gives $dom(ab)=y_1^{[t]}$. So, we see that  $y_1^{[t]}\in \Delta'\in Q_2$ because of $y_1\sim y_2$. Then $ab\in I_{P^*}\backslash T_{Q}$. If $rank(b)>rank(a)=rank(\alpha)$   then we obtain $ab\in I_{P^*}\backslash T_Q$ dually. If $rank(a),rank(b)>rank(\alpha)$ with $rank(ab)=rank(\alpha)$ then by Corollary \ref{>}, we have  $ab\notin I_{P^*}^u$. This shows $ab\in I_{P^*}\backslash T_{Q}$. Consequently, $I_{P^*}\backslash T_{Q}$ is a semigroup. 
 \end{proof}
 
 Now, we are able to characterize the maximal subsemigroups of $IOF_n^{par}$, which is the main result of this section.
 
\begin{thm} Let $J$ be a subsemigroup of $I_{P^*}$. Then $J$ is a maximal subsemigroup of $I_{P^*}$ if and only if  $J$ has one of the following forms. \\
\indent 1) $J=I_{P^*}\backslash T_{\{g\}}$ with  $|BC_{\{g\}}|=1$  for some $g\in P^*$ such that $\{g\}\in C$ and $T_{\{g\}}\subseteq I_{P^*}^u$; \\ 
\indent 2) $J=I_{P^*}\backslash T_{B_{y_1}}$ with $|BC_{B_{y_1}}|=1$ for some $y_1\in P^*$ such that $|y_1|\geq 2$, $B_{y_1}\in B $, and $T_{B_{y_1}}\subseteq I_{P^*}^u$; \\
\indent 3) $J=I_{P^*}\backslash T_{Q}$ for some partition $Q=( Q_1, Q_2 )$  of $BC_{\Delta_1}$, where $\Delta_1\in B \cup C$ and $T_Q\subseteq I_{P^*}^u$.   
\end{thm}

\begin{proof}
Let  $J$ be a maximal subsemigroup of $I_{P^*}$ and let $\alpha\in I_{P^*}\backslash J$. \\ \indent Assume  $\alpha\notin I_{P^*}^u$. Then by the definition of $I_{P^*}^u$, we get that $dom(\alpha)= y_2^{[r]}$ and $im(\alpha)= z_2^{[s]}$, where $y_2, z_2\in P^*$ and  $r,s\in\{1,2,...,|y_1|\}$ with $ y_2\sim y_1, z_2\sim z_1, y_1^{[r]}=z_1^{[s]}= y_1\cap z_1$, and  $y_2^{[r]}\sim z_2^{[s]}\sim y_1\cap z_1$ for some $y_1,z_1\in P^*$ with  $|y_1|=|z_1| \geq 2$. By Lemma \ref{dom,im}, we get that there are $\theta, \delta \in I_{P^*}$ with $rank(\theta)=rank(\delta)=|y_2|$ such that  $\theta\delta=\alpha$. Then $\theta, \delta \in J$ by Lemma \ref{beta}, a contradiction to $\alpha\notin J$. So, $\alpha\in I_{P^*}^u$.  \\
  \indent Suppose  that $m=dom(\alpha)$ and $m=im(\alpha)$ for all $m\in P^*$ with $ m \sim dom(\alpha) $.  So, we can put $m=dom(\alpha)=im(\alpha)$. This provides $\{m\}\in C$, i.e. $BC_{\{m\}}=\{\{m\}\}$.  So, $\alpha\in Id_{\overline{n}}$ and by the definition of $T_{\{m\}}$, we have that $T_{\{m\}}=\{\alpha\}$.   It is easy to see that  $J\cap \{\alpha\}=\emptyset$, this means $J\subseteq I_{P^*}\backslash \{\alpha\}$ and we have $I_{P^*}\backslash \{\alpha\}$ is a semigroup by Lemma \ref{alp}.  Together with $J$ is maximal subsemigroup of $I_{P^*}$, we have  $J=I_{P^*}\backslash \{\alpha\}$. \\
\indent Suppose there is $m \in P^*$ with $m\sim dom(\alpha)$ such that $ m \neq dom(\alpha)$ or $m \neq im(\alpha)$ and there are $y_1\in P^*$ with   $|y_1|\geq 2$ and $t\in\{1,2,...,|y_1|\}$ such that  for all $k\in P^*$ with $k\sim dom(\alpha)$  there is $y_3\in P^*$ with $y_3\sim y_1$ and $k=y_3^{[t]}$.  This implies that there exists $y_2\in P^*$ with $y_2\sim y_1\neq y_1$. Then $A_{y_1}\in A$ and thus, $B_{y_1}\in B$. We can conclude that $\{B_{y_1}\}=BC_{B_{y_1}}$ and $|B_{y_1}|\geq 2$.  So, we get $dom(\alpha), im(\alpha)\in B_{y_1}$. Then there are $y_4,y_5\in A_{y_1}$ such that  $dom(\alpha)=y_4^{[t]}$ and $im(\alpha)= y_5^{[t]}$. \\
\indent Assume there is $\theta\in J$ with $dom(\theta), im(\theta) \in B_{y_1}$. There are $y_6, y_7 \in A_{y_1}$ such that  $dom(\theta)=y_6^{[t]}, im(\theta)=y_7^{[t]}$. We have $\gamma_1,\gamma_2\in I_{P^*}$ with $dom(\gamma_1)=y_4, im(\gamma_1)=y_6, dom(\gamma_2)=y_7$, and $im(\gamma_2)=y_5$. This gives $\gamma_1, \gamma_2\in J$ because of $rank(\gamma_1), rank(\gamma_2)>rank(\alpha)$ together with Lemma \ref{beta}. We get that $\gamma_1\theta\gamma_2=\alpha$, a contradiction to $\alpha\notin J$ and $J$ is semigroup. Thus, $\theta\notin J$ for all $\theta\in I_{P^*}$ with $dom(\theta), im(\theta) \in B_{y_1}$, i.e. we have  $\theta\notin J$ for all $\theta\in T_{B_{y_1}}$.  This means $J\cap T_{B_{y_1}}=\emptyset$. So $J\subseteq I_{P^*}\backslash T_{B_{y_1}}$ and by Lemma \ref{yyy}, we have that $I_{P^*}\backslash T_{B_{y_1}}$ is a semigroup. Together with $J$ is maximal subsemigroup of $I_{P^*}$, we have  $J=I_{P^*}\backslash T_{B_{y_1}}$. \\
\indent Assume there is $\theta\in T_{B_{y_1}}$ with $\theta\notin I_{P^*}^u$.  By Corollary \ref{>}, there are $a_1,a_2\in I_{P^*}$ with $rank(a_1),rank(a_2)>rank(a_1a_2)$ such that $\theta =a_1a_2$. This provides, $a_1,a_2\in J$ by Lemma \ref{beta}, i.e. $\theta\in J$, a contradiction to $T_{B_{y_1}}\cap J = \emptyset$. Thus,  $T_{B_{y_1}}\subseteq I_{P^*}^u$. \\ 
\indent Suppose  for all $y_1\in P^*$ with $|y_1|\geq 2$ and for all $t\in\{1,2,...,|y_1|\}$, there is $k\in P^*$ with  $k\sim dom(\alpha)$ such that $k\neq y_3^{[t]}$ for all $y_3\in P^*$ with $y_3\sim y_1$, and  there is $m \in P^*$ with $m\sim dom(\alpha)$ such that $ m \neq dom(\alpha)$ or $m \neq im(\alpha)$.   Let $\lambda\in I_{P^*}\backslash J$ with $dom(\lambda)\sim dom(\alpha)$. Then we have the following four cases:

   1. $dom(\lambda)=y_3^{[t]}$ for some  $t\in\{1,2,...,|y_3|\}$ and $im(\lambda)=z_3^{[s]}$ for some  $s\in\{1,2,...,|z_3|\}$ and  $y_3,z_3\in P^*$ and there are $y_1,y_2,z_1,z_2\in P^*$ with $y_3\sim y_1\sim y_2\neq y_1$ and  $z_3\sim z_1\sim z_2\neq z_1$. \\
\indent Then there is $k\in P^*$ with $k\sim dom(\alpha)$ such that $k\neq y_3^{[t]}$. Assume $y_1\sim z_1$ and $t=s$.
 For $\theta\in I_{P^*}$, with $dom(\theta)=y_4^{[t]}, im(\theta)=z_4^{[s]}$, and $y_4\sim z_4\sim y_1$, we have $\theta\notin J$. Otherwise, let $\gamma_1,\gamma_2\in I_{P^*}$ with $dom(\gamma_1)=y_3, im(\gamma_1)=y_4, dom(\gamma_2)=z_4$, and $im(\gamma_2)=z_3$. This gives $\gamma_1, \gamma_2\in J$ because of Lemma \ref{beta}. We get that $\gamma_1\theta\gamma_2=\lambda$, a contradiction to $\lambda\notin J$ because $J$ is semigroup.   Moreover, there are $\beta_1, \beta_2\in I_{P^*}$ with $dom(\beta_1)=y_3^{[t]}, im(\beta_1)=dom(\beta_2)=k$, and $im(\beta_2)=z_3^{[s]}$. So, we observe that $\beta_1\beta_2=\lambda$. This show  $\beta_1\notin J$ or $\beta_2\notin J$. Suppose $\beta_1\notin J$. Let $\theta\in I_{P^*}$ with $dom(\theta)=z_3^{[s]}, im(\theta)=y_3^{[t]}$, where $y_3,z_3\in P^*$ with  $y_3\sim z_3\sim y_1$. As we have shown above, we have $\theta\notin J$. 
Note that $id\in J$ by Lemma \ref{beta} and $rank(\theta)<n$.
So $\beta_1\in \langle J \cup \{\theta\}\rangle$ and we  observe that $\beta_1=\alpha' \theta \rho$ with $\alpha' \in \langle J \cup \{\theta\} \rangle$ and $\rho\in J$. Since $y_3^{[t]}\neq k$, we can conclude that $\rho=\beta_1$, a contradiction to $\beta_1\notin J$.  Dually, we can prove that  $\beta_2 \notin J$ is not possible. Therefore, if $y_1\sim z_1$ then $t\neq s$.  \\ 
\indent \quad Furthermore, for all     $\beta\in I_{P^*} $ with $dom(\beta)=y_4^{[t]}$ and $im(\beta)=z_4^{[s]}$ for some $y_4,z_4\in P^*$ such that $y_4\sim y_1$ and  $z_4\sim z_1$, we have $\beta\notin J$. Otherwise, $\theta_1\beta\theta'=\lambda$, where $dom(\theta_1)=y_3, im(\theta_1)=y_4, dom(\theta')=z_4$, and $im(\theta')=z_3$, i.e. $\theta_1, \theta'\in J$ by Lemma \ref{beta}, a contradiction to $ \lambda\notin J$. \\
 
 2. $dom(\lambda)=y_3^{[t]}$  for some $t\in\{1,2,...,|y_3|\}$ with $y_3\in P^*$ and there are $y_1,y_2 \in P^*$ with $y_3\sim y_1\sim y_2\neq y_1$ and $im(\lambda)\not\sqsubset z$ for all $z\in P^*$ (or $im(\lambda)\sqsubset y\in P^*$ with $y\nsim z$ for all $z\in P^*\backslash \{y\}$).  \\ \indent Then for all     $\beta\in I_{P^*}$ with $dom(\beta)=y_4^{[t]}$ for some $y_4\in P^*$ with $y_4 \sim y_1$ and $im(\beta)=im(\lambda)$, we have $\beta\notin J$. Otherwise, $\theta\beta=\lambda$, where $dom(\theta)=y_3$ and $im(\theta)=y_4$ for some $\theta\in J$ by Lemma \ref{beta}, a  contradiction to $ \lambda\notin J$. \\
 
 3. $dom(\lambda)\not \sqsubset y$  for all $y\in P^*$ (or $dom(\lambda)\sqsubset z\in P^*$ with $z\nsim y$ for all $y\in P^*\backslash \{z\}$) and $im(\lambda)=z_3^{[s]}$ for some $s\in\{1,2,...|z_3|\}$,  $z_3\in P^*$, and there are $z_1,z_2\in P^*$ with $z_3\sim z_1\sim z_2\neq z_1$. \\ \indent Then for all     $\beta\in I_{P^*} $with $dom(\beta)=dom(\lambda)$ and $im(\beta)=z_4^{[s]}$ for some $z_4\in P^*, z_4\sim z_1$, we have $\beta\notin J$. Otherwise, $\beta\theta'=\lambda$ by Lemma \ref{beta}, where $dom(\theta')=z_4$ and $im(\theta')=z_3$ for some $\theta'\in J$, a contradiction to $ \lambda\notin J$. \\
 
 4. $dom(\lambda) \not\sqsubset y\in P^*$ for all $y\in P^*$
  (or $dom(\lambda)\sqsubset z\in P^*$ with $z\nsim y$ for all $y\in P^*\backslash \{z\}$) and $im(\lambda)\not\sqsubset y$ for all $y\in P^*$ (or $im(\lambda)\sqsubset z\in P^*$ with $z\nsim y$ for all $y\in P^*\backslash \{z\}$). \\

 \noindent Note that: $B_{y_3}, B_{z_3}\in B$ with $dom(\alpha)\in B_{y_3}$ and $im(\alpha)\in B_{z_3}$, if $\alpha=\lambda$ is of  form 1; \\
  $B_{y_3}\in B$ with $dom(\alpha)\in B_{y_3}$ and $\{im(\alpha)\}\in C$, if $\alpha=\lambda$ is of form 2;  \\
 $B_{z_3}\in B$ with $im(\alpha)\in B_{z_3}$ and $\{dom(\alpha)\}\in C$, if $\alpha=\lambda$ is of  form 3; \\
 $\{dom(\alpha)\}, \{im(\alpha)\}\in C$, if $\alpha=\lambda$ is of form  4. \\
Hence, there is $\Delta_1\in B\cup C$ with $dom(\alpha)\in \Delta_1$. We define, 
\begin{center} $\tilde{B}\tilde{C}=\{\Delta\in BC_{\Delta_1} : \text{there is} \ \beta\in I_{P^*}\backslash J \ \text{with} \ dom(\beta)\in \Delta\}$; \end{center} \begin{center}
 $\hat{B}\hat{C}=\{\Delta\in BC_{\Delta_1} : \text{there is} \ \beta\in I_{P^*}\backslash J \ \text{with} \ im(\beta)\in \Delta\}$; \end{center} \begin{center}
  $T_{BC}=\{ c \in I_{P^*} :  dom(c)\in \tilde{\Delta},im(c)\in \hat{\Delta} \ \text{with} \ \tilde{\Delta} \in \tilde{B}\tilde{C},  \hat{\Delta}\in \hat{B}\hat{C}\}$.  \end{center}
\indent Assume there are $\theta_1,\theta_2\notin J$ with $\theta_1\neq \theta_2$ and $im(\theta_1), dom(\theta_2)\in \Delta$ for some $\Delta\in BC_{\Delta_1}$. So, we observe that $\gamma\theta_1\rho=\theta_2$ with $\gamma\in \langle J \cup \{\theta_1\} \rangle$ and $ \rho\in J$. Then we have $im(\theta_2)\subset im(\rho)$ or $im(\theta_2)=im(\rho)$. Because of  $im(\theta_1), dom(\theta_2)\in \Delta$ and 1-3, then $im(\theta_2)=im(\rho)$. This gives $dom(\rho)=im(\theta_1)$. Then $\rho\notin J$ because of  2, a contradiction. Thus, $ \tilde{B}\tilde{C} \cap \hat{B}\hat{C}= \emptyset$. \\
 \indent Clearly, by the definition of $\tilde{B}\tilde{C}$ and $\hat{B}\hat{C}$, we have  $\tilde{B}\tilde{C}\ \cup \hat{B}\hat{C} \subseteq BC_{\Delta_1}$.  Let $\Delta\in BC_{\Delta_1}$. Further, let $h\in \Delta$, i.e. $h\sim dom(\alpha)$, and let $\gamma_1,\gamma_2 \in I_{P^*}$  with  $h=im(\gamma_1)=dom(\gamma_2), dom(\gamma_1)=dom(\alpha)$, and $im(\gamma_2)=im(\alpha)$. We have that $\gamma_1\gamma_2=\alpha$. This gives $\gamma_1\notin J$ or $\gamma_2\notin J$. So by the definition of $\hat{B}\hat{C}$($ \tilde{B}\tilde{C}$), we see that  if  $\gamma_1\notin J$( $\gamma_2\notin J$), then $\Delta\in \hat{B}\hat{C}$($\Delta\in \tilde{B}\tilde{C})$.  This means $\tilde{B}\tilde{C}\ \cup \hat{B}\hat{C} \supseteq BC_{\Delta_1}$ and together with $\tilde{B}\tilde{C}\ \cup \hat{B}\hat{C} \subseteq BC_{\Delta_1}$, we obtain $\tilde{B}\tilde{C}\ \cup \hat{B}\hat{C} = BC_{\Delta_1}$.   Consequently, $BC=(\tilde{B}\tilde{C}, \hat{B}\hat{C})$ is a partition of $BC_{\Delta_1}$. \\
\indent Let $\gamma \in T_{BC}$. Then there are $\Delta\in \tilde{B}\tilde{C}$ and $\Delta' \in \hat{B}\hat{C}$ such that $dom(\gamma)\in \Delta$ and $im(\gamma)\in \Delta'$. By the definition of   $\tilde{B}\tilde{C}$ and $\hat{B}\hat{C}$, there are $\delta_1,\delta_2\in I_{P^*}\backslash J$ with $dom(\delta_1)\in \Delta$ and $im(\delta_2)\in \Delta'$.  If $\gamma=\delta_1$ or $\gamma=\delta_2$ then we have $\gamma\notin J$.  Suppose $\gamma\neq\delta_1$ and $\gamma\neq\delta_2$. Recall, we have $dom(\gamma), dom(\delta_1)\in \Delta$. By 1-4, we get that there is $\theta_1\notin J$ with $dom(\theta_1)=dom(\gamma)$.  There is $\gamma'\in I_{P^*}$ with $dom(\gamma')=im(\gamma)$ and $im(\gamma')=im(\theta_1)$. We observe that $\gamma\gamma'=\theta_1$. This means, $\gamma\notin J$ or $\gamma'\notin J$.  Assume that $\gamma'\notin J$. We have  $dom(\gamma')=im(\gamma)\in \Delta'$. Then $\Delta'\in \tilde{B}\tilde{C}$ and we have $ \tilde{B}\tilde{C} \cap \hat{B}\hat{C}\neq \emptyset$, a contradiction.  Thus, $\gamma\notin J$. We can conclude that $T_{BC}\cap J = \emptyset$, this means $J\subseteq I_{P^*}\backslash T_{BC}$ and by Lemma \ref{bc}, we have $I_{P^*}\backslash T_{BC}$ is a semigroup. Together with $J$ is maximal subsemigroup of $I_{P^*}$, we have  $J=I_{P^*}\backslash T_{BC}$. \\
 \indent  Assume there is $\theta\in T_{BC}$ with $\theta\notin I_{P^*}^u$.  By Corollary \ref{>}, there are $a_1,a_2\in I_{P^*}$ with $rank(a_1),rank(a_2)>rank(a_1a_2)$ such that $\theta =a_1a_2$. This provides, $a_1,a_2\in J$ by Lemma \ref{beta}, i.e. $\theta\in J$, a contradiction to $T_{BC}\cap J = \emptyset$. Thus,  $T_{BC}\subseteq I_{P^*}^u$. \\

 Conversely, let $J=I_{P^*}\backslash T_{\{g\}}$ with  $|BC_{\{g\}}|=1$  for some $g\in P^*$ such that $\{g\}\in C$ and $T_{\{g\}}\subseteq I_{P^*}^u$. Then $I_{P^*}\backslash T_{\{g\}}$ is a semigroup by Lemma \ref{alp}. Since $| T_{\{g\}} |=1$, we can conclude that $J$ is a maximal subsemigroup of $I_{P^*}$. \\
 \indent Let $J=I_{P^*}\backslash T_{B_{y_1}}$ with $|BC_{B_{y_1}}|=1$ for some $y_1\in P^*$ such that $|y_1|\geq 2$, $B_{y_1}\in B $, and $T_{B_{y_1}}\subseteq I_{P^*}^u$. Then $I_{P^*}\backslash T_{B_{y_1}}$ is a semigroup by Lemma \ref{yyy}. Moreover, we can conclude that $\{B_{y_1}\}=BC_{B_{y_1}}$. 
Let $\alpha,\beta\in T_{B_{y_1}}$. There are $y_3,y_4,y_5,y_6\in A_{y_1}$ and some $t\in\{1,2,...,|y_1|\}$ such that $dom(\alpha)=y_3^{[t]}, im(\alpha)=y_4^{[t]}, dom(\beta)=y_5^{[t]}$, and $im(\beta)=y_6^{[t]}$. Further, there are $\theta_1,\theta_2\in I_{P^*}$ with $dom(\theta_1)=y_5, im(\theta_1)=y_3, dom(\theta_2)= y_4$, and $im(\theta_2)=y_6$. This shows $\theta_1,\theta_2\in J$ because of $rank(\theta_1)=rank(\theta_2)>rank(\alpha)$. So, we have $\theta_1\alpha\theta_2=\beta$. Thus, we get $\beta\in \langle J \cup \{\alpha\} \rangle$. Consequently, $J$ is a maximal subsemigroup of $I_{P^*}$. \\
\indent Let $J=I_{P^*}\backslash T_{Q}$ for some partition  $Q=( Q_1, Q_2 )$ of $BC_{\Delta_1}$, where $\Delta_1\in B \cup C$ and $T_Q\subseteq I_{P^*}^u$.  We have that $I_{P^*}\backslash T_{Q}$ is a semigroup by Lemma \ref{bc}.   Let $\alpha, \beta\in T_{Q}$. Then $dom(\alpha)\in \Delta_2, im(\alpha)\in \Delta_3, dom(\beta)\in \Delta_4$, and  $im(\beta)\in \Delta_5$, where $\Delta_2, \Delta_4\in Q_1$ and $\Delta_3,\Delta_5\in Q_2$. There are $\theta_1, \theta_2\in I_{P^*}$ with $dom(\theta_1)=dom(\beta), im(\theta_1)=dom(\alpha)$, $dom(\theta_2)=im(\alpha)$, and $im(\theta_2)=im(\beta)$. We get that $\theta_1, \theta_2\in J$ because of $im(\theta_1)=dom(\alpha)\in \Delta_2\in Q_1$ and $dom(\theta_2)=im(\alpha)\in \Delta_3\in Q_2$. So, we have $\theta_1\alpha\theta_2=\beta$. Thus, we get  $\beta\in \langle J \cup \{\alpha\} \rangle$.   Consequently, we have  that    $J$ is a maximal subsemigroup of $I_{P^*}$. 
 \end{proof}

\bibliography{sn-bibliography}

\begin{thebibliography}{9}
		
	\bibitem{clif}
	A. H. Clifford, G. B. Preston. The Algebraic Theory of Semigroups. Vol. I. Amer. Math. Soc. Surveys 7. Providence. R.I. (1961). Vol. II. Amer. Math. Soc. Surveys 7, Providence R.I. (1967)
	
	\bibitem{currie}
	J. D. Currie, T. I. Visentin. The number of order-preserving maps of fences and crowns. Order 8, 133-142(1991)
	
	\bibitem{fe} I. Dimitrova, V. H. Fernandes, J. Koppitz. The maximal subsemigroups of semigroups of transformations preserving or reversing the orientation on a finite chain. Publ. Math. Debrecen 81, 11–29(2012)   
	
	\bibitem{di} I. Dimitrova, J. Koppitz. On the maximal subsemigroups of some transformation semigroups. Asian-Eur. J. Math. 1, 189–202(2008) 
	
	\bibitem{II} I. Dimitrova, J. Koppitz. The Maximal Subsemigroups of the Ideals of Some Semigroups of Partial Injections. Discussiones Mathematicae, General Algebra and Applications 29, 153-167(2009)

\bibitem{I} I. Dimitrova, J. Koppitz. On the maximal regular subsemigroups of ideals of order-preserving or order-reversing transformations. Semigroup Forum 82,  172–180(2011)

	\bibitem{Dimi}
	I. Dimitrova,   J. Koppitz, L. Lohapan. Generating sets of semigroups of partial transformations preserving a zig-zag order on $\mathbb{N}$. International Journal of Pure and Applied Mathematics 17(2), 279-289(2017)
	
\bibitem{D} I. Dimitrova, T.  Mladenova. Classification of the Maximal Subsemigroups of the Semigroup of all Partial Order-Preserving Transformations. Union of Bulgarian Mathematicians 41.1, 158-162(2012). http://eudml.org/doc/250866

\bibitem{Jame}
J. East, J. Kumar, J. D. Mitchell, W. A. Wilson. Maximal subsemigroups of finite transformation and diagram monoids. Journal of Algebra 504, 176-216(2018)

	
	\bibitem{fer 2}
	V. H. Fernandes, J. Koppitz, T.  Musunthia. The rank of the semigroup of all order-preserving transformations on a finite fence. Bull. Malays. Math. Sci. Soc. 42,  2191-2211(2019)
			



\bibitem{Maz} O. Ganyushkin, V. Mazorchuk. On the Structure of $IO_n$. Semigroup Forum 66, 455–483(2003)

\bibitem{Gra} N. Graham, R. Graham, J. Rhodes. Maximal subsemigroups of finite semigroups. J. Combin. Theory 4, 203–209(1968)

	\bibitem{ig} I. Gyudzhenov, I. Dimitrova. On the maximal subsemigroups of the semigroup of all monotone transformations. Discussiones Mathematicae, General Algebra and Applications 26, 199–217(2006)

	\bibitem{Howie}
	J.M. Howie. Fundamentals of Semigroup Theory. Clarendon Press, Oxford (1995)
	
		
	\bibitem{Jen}
	K. Jendana, R. Srithus. Coregularity of order-preserving self-mapping semigroups of fence. Commun. Korean Math. Soc. 30, 349-361(2015)
	
	\bibitem{Jit}
	S. Jitman, R. Srithus, C. Worawannotai. Regularity of semigroups of transformations with restricted range preserving an alternating orientation order. Turkish Journal of Mathematics 42(4), 1913-1926(2018)

	
	\bibitem{Loh}
	L. Lohapan,  J. Koppitz. Regular semigroups of partial transformations preserving a fence $\mathbb{N}$. Novi Sad J. Math. 47, 77-91(2017)
	
	\bibitem{Rutkowski}
	A. Rutkowski. The formula for the number of order-preserving self-mappings of a fence. Order 9, 127-137(1992)
		
	
	\bibitem{Apa}
A. Sareeto, J. Koppitz,  (2023). The rank of the semigroup of order-, fence- and  parity-preserving partial  injections on a finite set. Asian-Eur. J. Math. https://doi.org/10.1142/S1793557123502236

\bibitem{Jane}
A. Sareeto, J. Koppitz,  (2023). A presentation for a submonoid of the  symmetric inverse monoid. https://arxiv.org/pdf/2310.15809.pdf 


\bibitem{Rat}
R. Srithus, R. Chinram, C.  Khongthat, Regularity in the Semigroup of Transformations Preserving a Zig-Zag Order. Bull. Malays. Math. Sci. Soc. 43 , 1761–1773(2020)

	\bibitem{Tanya}
	R. Tanyawong,  R. Srithus, R. Chinram. Regular subsemigroups of the semigroups of transformations preserving a fence. Asian-Eur. J. Math.  9(1), 1650003(2016)
	

	\bibitem{wagner}
	V. V. Wagner. Generalized groups. Dokl. Akad. Nauk SSSR 84, 1119–1122(in Russian) (1952)
	
\bibitem{Yan} X. Yang. A Classiﬁcation of Maximal Subsemigroups of Finite Order-Preserving Transfor-mation Semigroups. Communications in Algebra 28, 1503–1513(2000)


\bibitem{Zuh} 
P. Zhao,  HB. Hu. (2023). The ideals of the monoid of all orientation-preserving extensive full transformations. Journal of Algebra and Its Applications. https://doi.org/10.1142/S0219498825500318
	
\end{thebibliography}

\end{document}